\documentclass{article}
\usepackage{amssymb}
\usepackage{amsmath}
\usepackage{arxiv}
\usepackage{amsthm}
\usepackage{enumerate}
\usepackage{graphicx}
\usepackage{subcaption}
\usepackage[export]{adjustbox}
\usepackage{booktabs}
\usepackage{parskip}
\usepackage{xfrac}
\usepackage{textcomp}
\usepackage[utf8]{inputenc}

\usepackage{newunicodechar}

\DeclareUnicodeCharacter{0412}{\dash}
\DeclareRobustCommand\dash{%
  B}

\usepackage{url}            
\usepackage{amsfonts}       
\usepackage{csquotes}

\newtheorem{theorem}{Theorem}

\newtheorem{definition}{Definition}[section]

\newtheorem{proposition}{Proposition}

\title{ Commutator subgroup of Sylow 2-subgroups of alternating group and Miller-Moreno groups as bases of new Key Exchange Protocol }

\author{
Ruslan V. Skuratovskii \\
  NTUU 'Igor Sikorsky Kyiv Polytechnic Institute'\\
  \texttt{r.skuratovskii@kpi.ua}\\
  \texttt{ruslcomp@mail.ru} \\
   \And
 Aled Williams \\
  School of Mathematics\\
  Cardiff University\\
  Cardiff, UK\\
  \texttt{williamsae13@cardiff.ac.uk} \\
}

\begin{document}

\maketitle

\begin{abstract}
The goal of this investigation is effective method of key exchange which based on non-commutative group $G$. The results of Ko et al. \cite{kolee} is improved and generalized.
The size of a minimal generating set for the commutator subgroup of Sylow 2-subgroups of alternating group is found. The structure of the commutator subgroup of Sylow 2-subgroups of the alternating group ${A_{{2^{k}}}}$ is investigated and used in key exchange protocol which based on non-commutative group.

We consider non-commutative generalization of CDH problem
\cite{gu2013new, bohli2006towards} on base of metacyclic group  of Miller-Moreno type (minimal non-abelian group). We show that conjugacy problem in this group is intractable. Effectivity of computation is provided due to using groups of residues by modulo $n$. The algorithm of generating (designing) common key in non-commutative group with 2 mutually commuting subgroups is constructed by us.
\end{abstract}


\section{Introduction}
In this paper new conjugacy key exchange scheme is proposed. This protocol based on conjugacy problem in non-commutative group \cite{anshel1999algebraic, bohli2006towards, gu2013new, gu2014conjugacy, skuratovskii2019employment}. We slightly generalize Ko Lee’s \cite{kolee} protocol of key exchange. Public key cryptographic schemes based on the new systems are established. The conjugacy search problem in a group $G$ is the problem of recovering an $(a \in G)$ from given $(w \in G)$ and $h = a^{-1}wa$. This problem is in the core of several recently suggested public key exchange protocols. One of them is most notably due to Anshel, Anshel, and Goldfeld \cite{anshel1999algebraic} and another due to Ko et al. \cite{kolee}. As we know if CCP problem is tractable in $G$ then problem of finding $w^{ab}$ by given $w$, $w^{a}=a^{-1} w a$, $w^{b}=b^{-1} w b$ for an arbitrary fixed $w \in G$ such that is not from center of $G$, $w^{ab}$ is the common key that Alice and Bob have to generate.

Recently, a novel approach to public key encryption based on the algorithmic difficulty of solving the word and conjugacy problems for finitely presented groups has been proposed in \cite{anshel2001new, anshel1999algebraic}. The method is based on having a canonical minimal length form for words in a given finitely presented group, which can be computed rather rapidly, and in which there is no corresponding fast solution for the conjugacy problem. A key example is the braid group.

We denote by $w^x$ the conjugated element $u=x^{-1} w x$. We show that efficient algorithm that can distinguish between two probability distributions of $\left(w^{x}, w^{y}, w^{x y}\right)$ and $\left(w^{g}, w^{h}, w^{gh}\right)$ does not exist. Also, an efficient algorithm which recovers $w^{xh}$ from $w$, $w^x$ and $w^y$ does not exist. This group has representation $$G=\left\langle a, b | a^{p^{m}}=e, b^{p^{n}}=e, b^{-1} a b=a^{1+p^{m-1}}, m \geq 2, n \geq 1\right\rangle.$$
As a generators $a,b$ can be chosen two arbitrary commuting elements \cite{raievska2016finite, skuratovskii2019employment, otmani2010cryptanalysis}.

Consider non-metacyclic group of Millera Moreno. This group has representation
$$G=\left\langle a, b \big| | c | = p, |a| = p^m, |a| = p^n, m \ge 1, n \ge 1, b^{-1} ab = ac, b^{-1}cb = c \right\rangle.$$

To find a length of orbit of action by conjugation by $b$ we consider the class of conjugacy of elements of form $a^j c^i$. This class has length $p$ because of action $b^{-1} a^{j} c^{i} b=a^{j+1} c^{i}, \ldots,$ as well as $b^{-1} a^{j} c^{i+p-1} b=a^{j} c^{i+p}=a^{j} c^{i}$ increase the power of $c$ on 1. Thus, the first repetition of initial power $j$ in $a^j c^i$ occurs though $n$ conjugations of this word by $b$, where $1 \le j \le p$. Therefore, the length of the orbit is $p$.

We need to have an effective algorithm for computation of conjugated elements, if we want to design a key exchange algorithm based on non-commutative DH problem \cite{gu2014conjugacy}. Due to the relation in metacyclic group, which define the homomorphism $\varphi: \left\langle b \right\rangle \rightarrow \operatorname{Aut}(  \left\langle a \right\rangle)$ to the automorphism group of the $B =  \left\langle b \right\rangle$, we obtain a formula for finding a conjugated element. Using this formula, we can efficiently calculate the conjugated to element by using the raising to the $1 + p^{m-1}$-th power, where $m > 1$.

There is effective method of checking the equality of elements due to cyclic structure of group $A =  \left\langle a \right\rangle$ and $B =  \left\langle b \right\rangle$ in this group $G$.

We have an effective method of checking the equality of elements in the additive group $Z_n$ because of reducing by finite modulo $n$.

\section{Proof that conjugacy problem is $\mathcal{NP}$-hard in $G$. Size of a conjugacy class}
The orbit of the given base element $w \in G$ must must be long enough if we want to have problem of DL or equally problem of conjugacy in non-commutative group $G$ like $\mathcal{NP}$-hard problem.

Let elements of $G$ act by conjugation on $w \in G$, where $w \notin Z(G)$.

\begin{theorem}
The length of conjugacy class of non-central element $w$ is equal to $p$.
\end{theorem}

\begin{proof}
Recall the inner automorphism in $G$ is determined by the formula $b^{-1} a b=a^{1+p^{m-1}}$. Let us recall the structure of minimal non-abelian Metacyclic group, namely $G=B\ltimes_{\varphi} A$, where $A =  \left\langle a \right\rangle$ and $B =  \left\langle b \right\rangle$ are finite cyclic groups. Therefore, the formula $b^{-1}ab=a^{1+p^{m-1}}$ defines a homomorphism $\varphi$ in the subgroup of inner automorphisms $\operatorname{Aut}(  \left\langle a \right\rangle)$. It is well-known that each finite cyclic group is isomorphic to the correspondent additive cyclic group modulo $n$ residue $Z_n$. In this group equality of elements can be checked effectively due to reducing the elements of the module group.

Consider the orbit of element $w$ under action by conjugation. The length of such orbit can be found from equality $w^{(1+p^{m-1})^s} = w$ as minimal power $s$ for which this equality will be true. We apply Newton binomial formula to the expression $\left( 1+p^{m-1}\right) \equiv 1\left( \text{mod } p^{m}\right)$ and taking into account the relation $a^{p^m} = e$. We obtain
$$ 1 + C^1_s p^{m-1} + 1 + C^2_s p^{2(m-1)} + \dots + p^{s(m-1)} \equiv 1\left( \text{mod } p^{m}\right)$$
only if $s \equiv p^l (\text{mod } p^m)$ with $l < m$ because $1 + C^1_s p^{m-1} = 1 + sp^{m-1} \not\equiv 1 (\text{mod } p^s)$ if $s < p$. It means that the minimal $s$ when this congruence start to holds is equal to $p$. The prime number $p$ can be chosen as big as we need \cite{vinogradov2016elements} which completes the proof.
\end{proof}

Let us evaluate the size of subsets $S_1, S_2$ with mutually commutative elements. Each of this subset of generated by them subgroups $H_1, H_2$ can be chosen as the subgroups of center of group $G$. It is well-known that the semidirect product is closely related to wreath product. The center of the wreath product with non-faithful action were recently studied \cite{BDNU7575}.

\begin{proposition}
As it was proved by the author a center of the restricted wreath product with $n$ non-trivial coordinates $(A,X) \wr B$ is direct product of normal closure of center of diagonal of $Z(B^n)$, i.e. $(E \times Z(\Delta(B^n)))$, trivial an element, and intersection of $(K) \times E$ with $(A)$. In other words,
$$Z(\left( A,X\right) \wr B = \langle (1; \underbrace{h, h, \ldots, h}_{n}), e (Z(A) \cap Z(K, X)) \wr E \rangle \simeq \langle Z(A) \cap K ) \times Z(\Delta (B^n) \rangle$$
where $h \in Z(B), |X| = n$.

\end{proposition}

Taking into consideration that a semidirect product is the partial case of wreath product the diagonal of $B^n$ degenerates in $B$. Thus, we obtain such formula for the center of semidirect product:
$$Z\left( \left( A,X\right) \rtimes B\right) =\langle Z(1 ; h), e, (Z(A) \cap K, X) \wr E \rangle \simeq \langle Z(A) \cap K) \times Z(\Delta (B^n) \rangle.$$
This structure lead to constructive method of finding elements of the center. As it was noted above the elements $x$ and $y$ are parts of elements of secret key. Therefore as greater a size of center of a considered group as greater a size of a key space of this protocol.

Also commutator subgroup of Sylow 2-subgroup of alternating groups can be used as a support of CSP problem \cite{SkVin, SkJAA, SkAr}.
\begin{definition}  For an arbitrary  $k\in \mathbb{N}$ we call a $k$-\emph{coordinate} subgroup $U<G$  a subgroup, which is determined by $k$-coordinate sets $[U]_l$, $l\in \mathbb{N}$, if this subgroup consists of all Kaloujnine's tableaux $a\in I$
for which $[a]_l \in [U]_l$.
\end{definition}
We denote by ${{G}_{k}}(l)$ a level subgroup of $G_k$, which consists of the tuples of v.p. from ${{X}^{l}}$, $l<k-1$ of any $\alpha \in G_k$.

As a sets $S_1$ and $S_2$ consisting of mutually commutative elements we can use the set of elements of $l$-\emph{coordinate} subgroup of $G_k$, where $l<k$, or the elements of ${{G}_{k}}(l)$ that is isomorphic to this subgroup.
As it was proved by the author \cite{ SkVin} the order of ${Syl}_{2} A_{2^k}$ is ${{2}^{{{2}^{k}}-k-2}}$. therefore the growth of mutually commutative sets of elements $S_1$ and $S_2$ is exponential function has.

According to \cite{redei1947schiefe} index of center of metacyclic group has index $\left| G:Z\left( G\right) \right| =p^{2}$, therefore the order of $Z(G) = p^{k-2}$. Thus, we have $p^2 - 1$ possibilities to choose an element $w$ as an element of the open key, which is in the protocol of key exchange.

\section{Key exchange protocol}
Let $S_1, S_2$ be subsets from $G$ consisting of mutually commutative elements. We make a generalisation of CDH by taking into consideration the subgroups $H_1 =  \left\langle S_1 \right\rangle$ and $H_2 =  \left\langle S_2 \right\rangle$ instead of using $S_1, S_2$. We can do this because the groups $H_1$ and $H_2$ have generating sets $S_1$ and $S_2$ which commute. Because of these mutually commutative generating sets, we know that the subgroups are additionally mutually commutative.

\section{Consideration of base steps of the protocol}
\textit{Input: Elements $w$, $w^x$ and $w^y$}.

Alice selects a private $x$ as the random element $x$ from the subgroup $H_1$ and computes $w^{x}=x^{-1} w x$. The she sends it to Bob. Bob selects a private $y$ as the random element $y$ from the subgroup $H_2$ and computes $w^x$. Then he sends it to Alice. Bob computes $\left(w^{x}\right)^{y}=w^{x y}$ and Alice computes $\left(w^{y}\right)^{x}=w^{y x}$. Taking into consideration that $H_1$ and $H_2$ are mutually commutative groups we obtain that $xy = yx$. Therefore, we have that $w^{x y}=w^{y x}$.

\textit{Output: $w^{xy}$ that is the common key of Alice and Bob}.

Thus, the common key \cite{bohli2006towards, kolee, anshel1999algebraic, anshel2001new} $w^{xy}$ was successfully generated.

\textbf{Resistance to a cryptanalysis.} But if  an analytic use for a cryptanalysis  will use for cryptoanalysys solving of conjugacy search problem the method of reduction to solving of decomposition problem \cite{Shp}, then it lead us to solving of discrete logarithm problem in the multiplicative cyclic group ${{Z}_{p}}$. This problem is NP-hard for big  $p$.

\section{Conclusion}
We can choose mutually commutative $H_1, H_2$ as subgroups of $Z(G)$. As we said above, $x, y$ are chosen from $H_1, H_2$ as components of key. According to \cite{raievska2016finite} $Z(G) = p^{n+m-2}$ so size of key-space is $O(p^{n+m-2})$. It should be noted that the size of key-space can be chosen as arbitrary big number by choosing the parameters $p,n,m$. As an element for exponenting we can choose an arbitrary element $w \in A$ but $w \ne e$, because the size of orbit in result of action of inner automorphism $\varphi$ is always not less than $p$.


\end{document}